\documentclass[11pt,reqno]{amsart}
\numberwithin{equation}{section}
\usepackage{epsf, amsmath, amsfonts, amscd, amsthm, amssymb}

\theoremstyle{plain}
\newtheorem{theorem}{Theorem}[section]
\newtheorem{lemma}[theorem]{Lemma}
\newtheorem{proposition}[theorem]{Proposition}
\newtheorem{corollary}[theorem]{Corollary}

\theoremstyle{definition}
\newtheorem{definition}[theorem]{Definition}

\newtheorem{remark}[theorem]{Remark}

\newtheorem{notation}[theorem]{Notation}
\newtheorem{question}[theorem]{Question}

\newcommand{\gruppo}{{\mathbb{Z}}_3}
\newcommand{\first}{e^{2 \pi i/3}}
\newcommand{\third}{e^{-2 \pi i/3}}
\newcommand{\matrices}{M_{3}}
\newcommand{\Ad}{{\operatorname{Ad}}}

\newcommand{\id}{{\operatorname{id}}}

\newcommand{\omin}{\otimes_{\operatorname{min}}}

\newcommand{\af}{\alpha}
\newcommand{\bt}{\beta}
\newcommand{\gm}{\gamma}

\newcommand{\ep}{\varepsilon}

\newcommand{\et}{\eta}

\newcommand{\ld}{\lambda}
\newcommand{\sm}{\sigma}

\newcommand{\ph}{\varphi}

\newcommand{\om}{\omega}
\newcommand{\ta}{\tau}

\newcommand{\Z}{{\mathbb{Z}}}
\newcommand{\R}{{\mathbb{R}}}
\newcommand{\C}{{\mathbb{C}}}
\newcommand{\N}{{\mathbb{N}}}

\newcommand{\ca}{C*-algebra}
\newcommand{\tst}{tracial state}
\newcommand{\pj}{projection}
\newcommand{\hm}{homomorphism}

\newcommand{\andeqn}{\,\,\,\,\,\, {\mbox{and}} \,\,\,\,\,\,}
\newcommand{\dirlim}{\varinjlim}
\newcommand{\Aut}{{\operatorname{Aut}}}

\begin{document}
\title[C*-Algebra not Anti-isomorphic to Itself]{A
 Simple Separable Exact C*-Algebra
 not Anti-isomorphic to Itself}
\author{N.~Christopher Phillips and Maria Grazia Viola}
\address{\hskip-\parindent Department of Mathematics\\
 University of Oregon\\
 Eugene OR 97403-1222 \\
 USA}
\email{ncp@darkwing.uoregon.edu}
\address{\hskip-\parindent
 Department of Mathematics and Interdisciplinary Studies \\
 Lakehead University-Orillia \\
 1~Colborne Street~W \\
 Orillia ON L3V 7X5 \\
 Canada}
\email{mviola@lakeheadu.ca}
\date{17~December 2011}

\thanks{N.~Christopher Phillips was partially supported
by NSF grant DMS-0701076,
by the Fields Institute for Research in Mathematical Sciences,
Toronto, Canada,
and by an Elliott Distinguished Visitorship at the Fields Institute.
Maria Grazia Viola was partly supported by the Fields Institute for
Research in Mathematical Sciences,
and by Queen's University, Kingston, Canada}

\begin{abstract}
We give an example of an exact, stably finite, simple,
separable \ca\  $D$ which is not isomorphic to its opposite algebra.
Moreover, $D$ has the following additional properties.
It is stably finite,
approximately divisible,
has real rank zero and stable rank one,
has a unique tracial state,
and the order on \pj s over~$D$ is determined by traces.
It also absorbs the Jiang-Su algebra~$Z,$
and in fact absorbs the $3^{\infty}$~UHF algebra.
We can also explicitly compute the K-theory of~$D,$
namely $K_0 (D) \cong {\mathbb{Z}} \big[ \tfrac{1}{3} \big]$
with the standard order,
and $K_1 (D) = 0,$ as well as the Cuntz semigroup of $D,$
namely
$W (D)
 \cong {\mathbb{Z}} \big[ \tfrac{1}{3} \big]_{+} \sqcup (0, \infty).$

\end{abstract}

\maketitle

\section{Introduction}\label{Sec:1}

A \ca\  $A$ is said to be anti-isomorphic to itself if
is isomorphic to its opposite algebra $A^{\mathrm{op}}.$
The opposite algebra is by definition the \ca\  %
whose underlying vector space structure, norm, and adjoint
are the same as for $A,$
while the product of $x$ and $y$ is equal to $y x$ instead of $x y.$

In this paper we give an example of a simple separable exact \ca\  %
which is not anti-isomorphic to itself.
There are several examples in the literature
of factors of type~${\mathrm{II}}_1$
and type~${\mathrm{III}}$
 which are not isomorphic to their opposite algebras.
(See for example \cite{Connes1}, \cite{Viola}, and~\cite{Connes2}.)
A \ca\  isomorphism of von Neumann algebras is necessarily a
von Neumann algebra isomorphism, by Corollary~5.13 of~\cite{SZ}.
These are therefore also examples of simple \ca s
not isomorphic to their opposite algebras.
However, none of these examples is separable or exact
in the \ca\  sense.
(An infinite factor contains a copy of every separable \ca,
so is not exact as a \ca.
A factor of type~${\mathrm{II}}_1$ contains a copy
of the \ca~$M$
of all bounded sequences in $\prod_{k = 1}^{\infty} M_k,$
and the proof of the proposition in Section~2.5 of~\cite{Ws}
shows that $M$ is not exact.)

In a recent work~\cite{Phil}, the first author proved
the existence of a simple separable \ca~$A$
which is not isomorphic to its opposite algebra.
Moreover, $A$ was shown to have real rank zero and $K_1 (A) = 0.$
The \ca~$A$ was constructed as a subalgebra
of Connes' example of a ${\mathrm{II}}_1$~factor $N$
not anti-isomorphic to itself (see~\cite{Connes1}),
by finding a series of intermediate \ca s $(B_n)_{n \geq 0}$
in $N$ with some nice properties
relative to the trace and the real rank,
and taking $A = {\overline{\bigcup_{n = 0}^{\infty} B_n}}.$
However, the construction did not give
an explicit model for this example.
Problem~4.5 of~\cite{Phil} asked for a more natural
and explicit example of a simple separable \ca\  %
not isomorphic to its opposite algebra,
and Problem~4.7 of~\cite{Phil} asked for an exact example.
We provide an example satisfying both conditions here.

We construct our example using a C*~analog of a variation,
due to the second author~\cite{Viola},
of Connes' example of a ${\mathrm{II}}_1$~factor
not anti-isomorphic to itself.
We let $\Z_n$ denote $\Z / n \Z.$
Our \ca\  is a crossed product of a \ca~$C$
by an action $\gamma$ of $\gruppo.$
We show that $D = C \rtimes_{\gamma} \gruppo$
is exact, stably finite, simple, separable, and unital.
Moreover, using our explicit construction,
we show that $D$
is approximately divisible and has stable rank one and real rank zero,
and that it tensorially absorbs the Jiang-Su algebra~$Z$
and the $3^{\infty}$~UHF algebra.
Also, the order on projections over $D$ is
determined by traces,
and the K-groups and Cuntz semigroup of $D$ are given by
\[
K_0 (D)
 \cong {\mathbb{Z}} \big[ \tfrac{1}{3} \big],
\,\,\,\,\,\,
K_1 (D) = 0,
\andeqn
W (D)
 \cong {\mathbb{Z}} \big[ \tfrac{1}{3} \big]_{+} \sqcup (0, \infty).
\]

In Section~\ref{Sec:2} we recall definitions related to
the properties we prove for~$D.$
Section~\ref{Sec:3} contains a result
on tracial states on crossed products.
The \ca~$C$ and the action $\gamma$
are described in Section~\ref{Sec:4}.
The algebra~$C$ is the tensor product of a reduced free product~$A$
and the $3^{\infty}$~UHF algebra~$B.$
The action $\gamma$ is a perturbation by an inner automorphism
of the tensor product of two automorphisms $\alpha \in \Aut (A)$
and $\beta \in \Aut (B).$
The proof that $D \not\cong D^{\mathrm{op}},$
which is done in Section~\ref{Sec:5},
relies on uniqueness of the tracial state on~$D,$
and the fact that its Gelfand-Naimark-Segal representation
yields a ${\mathrm{II}}_1$~factor
not isomorphic to its opposite algebra.
We show in Section~\ref{Sec:6} that $\gamma$ has the Rokhlin property,
which is the key to obtaining the additional
properties of $D$ that are given above.
These are proved in Sections \ref{Sec:6} and~\ref{Sec:7}.
We conclude with some open questions in Section~\ref{Sec:8}.

\section{Definitions}\label{Sec:2}

In this section,
we give some definitions
and known results which play a key role in this paper.

\begin{definition}
A \ca\  $A$ is {\emph{exact}} if for every short exact sequence of \ca s
and homomorphisms
\[
0 \longrightarrow J
 \longrightarrow B
 \longrightarrow B/J
 \longrightarrow 0,
\]
the sequence of spatial tensor products
\[
0 \longrightarrow A \omin J
 \longrightarrow A \omin B
 \longrightarrow A \omin (B/J)
 \longrightarrow 0
\]
is exact.
\end{definition}

The theory of exact \ca s was developed by Kirchberg
in a series of papers,
particularly \cite{Kr0.5}, \cite{Kr1}, and~\cite{Kr2}.
He proved a number of conditions equivalent to exactness
for separable, unital \ca s.
For example, exactness
is equivalent to nuclear embeddability.
(See Theorem~4.1 and the preceding comment in~\cite{Kr1}.)
He also proved that the class of exact \ca\  is closed
under some of the standard operations on \ca s,
like quotients, tensor products, direct limits,
and crossed products by amenable locally compact groups
(Proposition~7.1 of~\cite{Kr2}).
Nuclear \ca s are exact, so abelian \ca s, AF~algebras,
and group \ca s of amenable groups are all exact.

\begin{definition}\label{D:MvN}
Let $A$ be a \ca,  and let $p, q \in A$ be projections.
Then $p \precsim q$ if and only if there is $v \in A$ such that
$p = v v^*$ and $v^* v \leq q.$
\end{definition}

In a finite factor,
the order on projections is determined by the trace.
This means that if $M$ is a finite factor
with standard trace $\ta,$
and $e, f \in M$ are \pj s,
then $e \precsim f$ if and only if $\tau (e) \leq \tau (f).$
In~\cite{Bla} Blackadar asked to what extent a
similar comparison theory could be developed for simple unital \ca s.
A simple unital \ca\  $A$ may have more than one \tst,
so we use the set $T (A)$ of all tracial states on~$A.$
The following definition
is Blackadar's Second Fundamental Comparability Question
for $\bigcup_{n = 1}^{\infty} M_n (A).$
See 1.3.1 in~\cite{Bla}.

\begin{definition}\label{D:OrdDetTr}
Let $A$ be a unital \ca.
Set $M_{\infty} (A) = \bigcup_{n = 1}^{\infty} M_n (A),$
using the usual embedding of $M_n (A)$ in $M_{n + 1} (A)$
as the upper left corner.
We say that
{\emph{the order on projections over $A$ is determined by traces}}
if whenever $p, q \in M_{\infty} (A)$ are projections
with $q \neq 0$ such that
$\tau (p) < \tau (q)$ for every $\tau$ in $T (A),$
then $p \precsim q.$
\end{definition}

Unfortunately,
one must sometimes use quasitraces instead of tracial states.
A quasitrace on a \ca~$A$ is essentially a not necessarily linear trace,
but which extends to all matrix algebras over~$A.$
See 4.2 of~\cite{Rordam2}.
(This definition differs slightly from that in
Definition II.1.1 of~\cite{BlaHan},
which omits normalization and the requirement that
the quasitrace extend properly to matrix algebras.)
It is an open question whether every quasitrace is a trace.

\begin{definition}[Definition~1.2 of~\cite{BlaRor}]\label{D:AppD}
Let $A$ be a separable unital \ca.
We say that $A$ is {\emph{approximately divisible}} if for every
$x_1, x_2, \ldots, x_n \in A$ and $\varepsilon > 0,$
there exists a finite-dimensional C*-subalgebra $B$ of $A,$
containing the unit of $A,$ such that:
\begin{enumerate}
\item\label{D:AppD:1}
$B$ has no commutative central summand, that is,
$B$ has no abelian central projection.
\item\label{D:AppD:2}
$\| x_k y - y x_k \| < \varepsilon$ for $k = 1, 2, \ldots, n$
and all $y$ in the unit ball of~$B.$
\end{enumerate}
\end{definition}

Many standard simple \ca s are approximately divisible.
For our purposes, we need:

\begin{lemma}\label{L:AppDivAlgs}
\begin{enumerate}
\item\label{L:AppDivAlgs:1}
Every infinite-dimensional simple unital AF~algebra
is approximately divisible.
\item\label{L:AppDivAlgs:2}
Let $A$ and $B$ be separable unital \ca s.
If $A$ is approximately divisible,
then so is the tensor product $A \otimes B,$
with any choice of tensor norm.
\end{enumerate}
\end{lemma}

\begin{proof}
Part~(\ref{L:AppDivAlgs:1}) is contained in
Proposition~4.1 of~\cite{BlaRor}.
Part~(\ref{L:AppDivAlgs:2}) is immediate.
\end{proof}

\begin{proposition}\label{order}
Let $A$ be a
finite approximately divisible simple separable unital exact \ca.
Then the order on projections over $A$ is determined by traces.
\end{proposition}

\begin{proof}
Corollary 3.9(b) of~\cite{BlaRor} implies that
if $p, q \in M_{\infty} (A)$ are \pj s
with $q \neq 0$ and $\tau (p) < \tau (q)$ for
every quasitrace $\ta$ on~$A,$
then $p \precsim q.$
It is proved in~\cite{Haa} that any quasitrace
on a unital exact \ca\  is a tracial state.
It is now clear that
if $\tau (p) < \tau (q)$ for
every quasitrace $\ta$ on~$A,$
then $q \neq 0.$
This verifies Definition~\ref{D:OrdDetTr}.
\end{proof}

\begin{remark}\label{R:HQT}
Since~\cite{Haa} remains unpublished,
we explain how one can prove
Proposition~\ref{order} using only published results.
First, every quasitrace $\ta$ on~$A$
defines a state on $K_0 (A)$ by using the extension
of $\ta$ to $M_{\infty} (A)$ and the formula
$\ta_* ([p]) = \ta (p).$
The converse is also true
(Theorem~3.3 of~\cite{BR}).
So it is enough to show that every state on $K_0 (A)$
is actually induced by a tracial state on~$A.$
For a unital exact \ca~$A,$ this is Corollary~9.18 of~\cite{HT}.
\end{remark}

We follow Section~2 of~\cite{BPT}
(not the original source)
for the Cuntz semigroup.
See~\cite{BPT} and the references there
for proofs of the assertions made here.

\begin{definition}\label{D:CuEq}
Let $A$ be a \ca.
Given $a, b\in M_{\infty} (A)_{+},$
we say that $a$ is {\emph{Cuntz subequivalent}} to~$b,$
denoted by $a \precsim b,$
if there is a sequence $(v_n)_{n = 1}^{\infty}$
in $M_{\infty} (A)$ such that
\[
\lim_{n \to \infty} \|v_n b v_n^* - a \| = 0.
\]
Moreover, we say that $a$ and $b$ are {\emph{Cuntz equivalent}},
denoted by $a \sim b,$ if $a \precsim b$ and $b \precsim a.$
\end{definition}

This is an equivalence relation,
so we can make the following definition.

\begin{definition}\label{D:CuSg}
Given a \ca~$A,$ we
define the Cuntz semigroup $W(A)$ of~$A$
to be $M_{\infty} (A)_{+} / {\sim}.$
We write $\langle a \rangle$
for the class of $a \in M_{\infty} (A)_{+}.$
We define a semigroup operation by
$\langle a \rangle + \langle b \rangle = \langle a \oplus b \rangle,$
and a partial order by
$\langle a \rangle \leq \langle b \rangle$
if and only if $a \precsim b.$
\end{definition}

This semigroup is an analog for positive elements
of the semigroup
of Murray-von Neumann equivalence classes of projections.
If the \ca~$A$ is stably finite,
then two projections $p, q \in M_{\infty} (A)$
are Murray-von Neumann equivalent if and only if
they are Cuntz equivalent.
In general, it is not easy to compute the Cuntz semigroup
and it is not known how the Cuntz semigroup behaves
with respect to tensor products or reduced free products.

\section{Uniqueness of the tracial state
 on some crossed products}\label{Sec:3}

Let $\alpha \colon G \to \Aut (A)$
be an action of a finite group~$G$ on a unital \ca~$A.$
Consider the crossed product
$A \rtimes_{\alpha} G = A \rtimes_{\alpha, {\mathrm{r}}} G,$
with standard unitaries $u_g$ coming from the elements of~$G.$
The usual conditional expectation
$E \colon A \rtimes_{\alpha} G \to A$
is given by $E \left (\sum_{g \in G} a_{g} u_{g} \right) = a_1,$
where $1$ is the identity of~$G.$
Any state $\om$ on $A$ gives rise to a state $\om \circ E$
on $A \rtimes_{\alpha} G.$

\begin{remark}\label{R:Inv}
Let the notation be as above, and further assume $G$ is abelian.
One easily checks that a state $\psi$ on $A \rtimes_{\alpha} G$
is of the form $\om \circ E$ for some state $\om$ on $A$
if and only if $\psi$ is invariant under the dual action
${\widehat{\alpha}} \colon {\widehat{G}}
        \to {\operatorname{Aut}} (A \rtimes_{\alpha} G).$
\end{remark}

\begin{lemma}\label{utrace}
Let $A$ be a unital \ca\  %
with a unique tracial state $\tau.$
Let $\alpha$ be an action of ${\mathbb{Z}}_{n}$ on $A.$
Let $\pi_{\tau}$ be the Gelfand-Naimark-Segal representation of $A$
corresponding to $\tau.$
Set
$N_{\tau}
 = (\pi_{\tau} (A))^{\prime\prime} \subseteq B (L^2 (A, \tau)).$
Denote by $\widetilde{\alpha}$ the extension of $\alpha$ to $N_{\tau}.$
Assume that the von Neumann algebra
$N_{\tau} \rtimes_{\widetilde{\alpha}} {\mathbb{Z}}_n$
is a ${\mathrm{II}}_1$ factor.
Then there is a unique tracial state
on $A \rtimes_{\alpha} {\mathbb{Z}}_{n}.$
\end{lemma}

\begin{proof}
Let $E \colon A \rtimes_{\alpha} {\mathbb{Z}}_{n} \to A$
be the conditional expectation on the crossed product,
as defined above.
Then $\tau \circ E$ is a tracial state
on $A \rtimes_{\alpha} {\mathbb{Z}}_{n}.$
Suppose $\sm$ is a different tracial state
on $A \rtimes_{\alpha} {\mathbb{Z}}_{n}.$
Identifying ${\widehat{\mathbb{Z}}_n}$ with ${\mathbb{Z}}_{n},$
let
${\widehat{\alpha}} \colon {\mathbb{Z}}_{n} \to
 \Aut (A \rtimes_{\alpha} {\mathbb{Z}}_{n})$
denote the dual action.
Take
\[
\mu = \frac{1}{n} \sum_{k = 0}^{n - 1} \sm \circ {\widehat{\alpha}}_k.
\]
Then $\mu$ is an ${\widehat{\alpha}}$-invariant tracial state,
so by Remark~\ref{R:Inv} and uniqueness of~$\ta,$
we have $\mu = (\mu |_{A}) \circ E = \tau \circ E.$
Therefore.
\[
\sm \leq n \cdot (\tau \circ E).
\]

We claim that $\sm$ extends to a
normal tracial state ${\widetilde{\sm}}$
on
$N_{\tau} \rtimes_{\widetilde{\alpha}} {\mathbb{Z}}_{n}.$
Denote by $\tau_0$ the unique trace on
$N_{\tau} \rtimes_{\widetilde{\alpha}} {\mathbb{Z}}_{n},$
so that $\tau_0 |_{A \rtimes_{\alpha} {\mathbb{Z}}_{n}} = \tau \circ E.$
Recall that the $L^2$-norm on
$N_{\tau} \rtimes_{\widetilde{\alpha}} {\mathbb{Z}}_{n}$
is given by $\| y \|_2^2 = \tau_0 (y^* y)$
for $y \in N_{\tau} \rtimes_{\widetilde{\alpha}} {\mathbb{Z}}_{n}.$
Fix $x \in N_{\tau} \rtimes_{\widetilde{\alpha}} {\mathbb{Z}}_{n},$
and choose a sequence $(x_m)_{m \geq 0}$
in $A \rtimes_{\alpha} {\mathbb{Z}}_{n}$
such that $\lim_{m \to \infty} \| x_m - x \|_2 = 0.$
Then for every $\varepsilon > 0$ there exists $n_0$
such that
$\| x_k - x_l \|_2 \leq \varepsilon / \sqrt{n}$ for $k, l > n_0.$
Using the Cauchy-Schwarz inequality
for the inner product corresponding to the tracial state $\sm$
at the first step,
we get
\begin{align*}
| \sm (x_l - x_k) |^2
 & \leq \sm ((x_l - x_k)^* (x_l - x_k))
 \leq n \big[ \tau \circ E ((x_l - x_k)^* (x_l - x_k)) \big]  \\
 & = n \|x_l - x_k\|_2^2
 \leq \varepsilon^2
\end{align*}
for all $k, l > n_0.$
So the sequence $( \sm (x_m) )_{m \geq 0}$ is Cauchy.
Set ${\widetilde{\sm}} (x) = \lim_{m \to \infty} \sm (x_m).$
By a similar argument we can show
that the definition of ${\widetilde{\sm}} (x)$
does not depend on the choice of the sequence $(x_m)_{m \geq 0}.$
This proves the claim.

Since ${\widetilde{\sm}}$
is a normal trace on
$N_{\tau} \rtimes_{\widetilde{\alpha}} {\mathbb{Z}}_{n}$
satisfying ${\widetilde{\sm}} \leq n \tau_0,$
there exists an element $c$ in the center of
$N_{\tau} \rtimes_{{\widetilde{\alpha}}} {\mathbb{Z}}_{n}$
such that:
\begin{itemize}
\item
$0 \leq c \leq n.$
\item
$\tau_0 (c) = 1.$
\item
${\widetilde{\sm}} (x) = \tau_0 (x c)$
for every
$x \in N_{\tau} \rtimes_{\widetilde{\alpha}} {\mathbb{Z}}_n.$
\end{itemize}
The element $c$ is not a scalar since
${\widetilde{\sm}} \neq \ta_0.$
This contradicts the assumption that
$N_{\tau} \rtimes_{\widetilde{\alpha}} {\mathbb{Z}}_n$ is a factor.
\end{proof}

\section{The description of
 the \ca\  $D = C \rtimes_{\gamma} \gruppo$}\label{Sec:4}

In this section we describe our example of a \ca~$D$
not isomorphic to its opposite algebra.
It is the crossed product of a \ca~$C$
by an action $\gamma$ of the finite group $\gruppo.$
The algebra $C$ is the tensor product $A \otimes B,$
where $A$ is a reduced free product \ca\  %
and $B$ is a UHF~algebra.
The automorphism $\gamma$ which generates our action of $\gruppo$
is a perturbation, by an inner automorphism,
of an automorphism of the form $\alpha \otimes \beta,$
with $\alpha \in \Aut (A)$ and $\beta \in \Aut (B).$
We start by describing the two \ca s $A$ and~$B,$
and the two automorphisms $\alpha$ and~$\beta.$

\begin{definition}\label{D:B0}
Define $\varphi_n \colon M_{3^n} \to M_{3^{n + 1}}$ by
$\varphi_n (x) = {\operatorname{diag}} (x, x, x)$
for $x \in M_{3^n}.$
Let $B$ be the UHF~algebra obtained as the direct limit
of the sequence $(M_{3^n}, \varphi_n)_{n \in \N}.$
We identify $M_{3^n}$ with the tensor product
$\matrices \otimes \matrices \otimes \cdots \otimes \matrices$
($n$ copies of $\matrices$)
and think of $B$ as the infinite tensor product \ca\  %
$\bigotimes_1^{\infty} \matrices.$
\end{definition}

\begin{notation}\label{N:B0}
Denote by $B_n$ the image of the embedding
$j_n \colon M_{3^n} \to B,$ which we identify with
\[
\matrices \otimes \matrices \otimes \cdots \otimes \matrices
 \otimes 1 \otimes 1 \otimes \cdots
\]
($n$ copies of $\matrices$).
For $k \geq 1,$ let $\pi_{k} \colon \matrices \to B$ be the map
\[
\pi_k (x) = 1 \otimes 1 \otimes \cdots \otimes 1 \otimes x
 \otimes 1 \otimes 1 \otimes \cdots,
\]
with $x$ in position~$k,$
and denote by $\lambda$ the shift endomorphism of~$B,$
determined by $\lambda (\pi_k (x)) = \pi_{k + 1} (x)$
for all $k$ and all $x \in\matrices.$
Let $(e_{i, j})_{i, j = 1}^3$
be the standard system of matrix units in $\matrices.$
Define unitaries $u, v \in B$ by
\[
v = \pi_1
 \big( e^{2 \pi i/3} e_{1,1} + e^{4 \pi i/3} e_{2, 2} + e_{3, 3} \big)
\]
and
\[
u = \pi_1 (e_{3, 1}) \lambda (v^*) + \pi_1 (e_{1, 2})
   + \pi_1 (e_{2, 3}).
\]
\end{notation}

{}From the above data we construct an automorphism $\bt$ of $B$
of outer period~$3,$
that is, $3$ is the least integer $m > 0$ such that $\bt^m$ is inner.
The following result is due to Connes
(Proposition~1.6 of~\cite{Connes3}),
who proved it in the context
of convergence in the strong operator topology
by working with the $L^2$-norm.
Below we simply outline the few changes
from the von Neumann algebra setting to the \ca\  setting.

\begin{lemma}[Connes]\label{lemma}
Let the notation be as in Definition~\ref{D:B0} and Notation~\ref{N:B0}.
Define automorphisms $\bt_n \in \Aut (B)$ by
\[
\beta_n (x)
 = \Ad \big( u \lambda (u) \lambda^2 (u) \cdots \lambda^n (u) \big) (x)
\]
for all $x \in B.$
Then $\beta (x) = \lim_{n \to \infty} \beta_n (x)$
exists for every $x \in B,$
and $\beta$ is an outer automorphism of~$B$
such that $\beta^3 = \Ad (v)$ and $\beta (v) = e^{2 \pi i/3} v.$
\end{lemma}

\begin{proof}
For each fixed $k \geq 1,$ and for any $j > 0,$
the element $\lambda^{k + j}(u)$
commutes with all elements of $B_k.$
Thus $\beta_{k + j} (x) = \beta_k (x)$
for any $j > 0$ and $x \in B_k.$
Since $\bigcup_{k = 0}^{\infty} B_k$ is dense in~$B,$
it follows that $\beta (y) = \lim_{n \to \infty} \beta_{n} (y)$
exists for all $y \in B.$
This gives a homomorphism $\bt \colon B \to B,$
necessarily injective since $B$ is simple.
Clearly $\beta (v) = u v u^* = \first v.$
Next, the proof of Proposition~1.6 of~\cite{Connes3}
shows that $\beta^3 (y) = \Ad (v) (y)$ for every $y \in B.$
The equality $\beta^3 (y) = \Ad (v) (y)$
implies that $\beta$ is surjective,
so $\beta$ is an automorphism of~$B.$
Finally, $\beta$ is outer because Proposition~1.6 of~\cite{Connes3}
implies that it is outer on the type ${\mathrm{II}}_1$~factor
obtained from the Gelfand-Naimark-Segal construction
using the unique tracial state on~$B.$
\end{proof}

\begin{definition}\label{D:A0}
Define the \ca\  $A$ to be the reduced free product
\[
A = C ([0, 1]) \star_{\mathrm{r}} C ([0, 1])
 \star_{\mathrm{r}} C ([0, 1])
 \star_{\mathrm{r}} {\mathbb{C}}^3,
\]
amalgamated over~$\C,$
taken with respect to the states given by Lebesgue measure $\mu$
on the first three factors and
the state given by
$\ph (c_1, c_2, c_3) = \tfrac{1}{3} (c_1 + c_2 + c_3)$
on the last factor.
For $k = 1, 2, 3,$ let $\varepsilon_{k} \colon  C ([0, 1]) \to A$
be the inclusion of the $k$th copy of $C ([0, 1])$
in the reduced free product.
Let $u_0 = (\first, \, 1, \, \third) \in \C^3,$
and regard $u_0$ as an element of $A$ via the obvious inclusion.
\end{definition}

\begin{lemma}\label{L:A0}
The \ca\  $A$ of Definition~\ref{D:A0} is unital, simple, separable,
exact, and has a unique \tst.
\end{lemma}

\begin{proof}
It is trivial that $A$ is unital and separable.
That it is simple and has a unique \tst\   follows from
several applications of the corollary on page~431 of~\cite{Av}.
(The notation is as in Proposition~3.1 of~\cite{Av},
and the required unitaries are easy to find.)
Exactness follows from Theorem~3.2 of~\cite{Dykema}.
\end{proof}

\begin{lemma}\label{L:af}
Let the notation be as in Definition~\ref{D:A0}.
Then there exists a unique automorphism $\alpha \in \Aut (A)$
such that,
for all $f \in C ([0, 1]),$
\[
\alpha ( \varepsilon_{1} (f)) = \varepsilon_{2} (f),
\,\,\,\,\,\,
\alpha ( \varepsilon_{2} (f)) = \varepsilon_{3} (f),
\andeqn
\alpha ( \varepsilon_{3} (f)) = \Ad (u_0) ( \varepsilon_{1} (f)),
\]
and such that $\alpha (u_0) = \third u_0.$
Moreover, $\alpha^3 = \Ad (u_0).$
\end{lemma}

\begin{proof}
We can identify ${\mathbb{C}}^3$ with the
universal \ca\  generated by a unitary of order~$3,$
with the generating unitary taken to be~$u_0.$
The definition of $\alpha (u_0)$ is then legitimate
because $\big( \third u_0 \big)^3 = 1.$
It is now obvious that
if we replace $A$ by the full free product
$C ([0, 1]) \star C ([0, 1]) \star C ([0, 1]) \star {\mathbb{C}}^3$
(taking $u_0$ and the values of $\ep_k$ to be in the full free product),
then the definition gives an endomorphism $\af_0$
of the full free product.
Moreover, one checks immediately that $\af_0^3 = \Ad (u_0)$
on the full free product.
Finally,
using $\ph (u_0) = 0,$
it is easily checked that $\af_0$ preserves the
free product state
$\mu \star \mu \star \mu \star \ph.$
(See Proposition~1.1 of~\cite{Av} for the definition
of the free product of two states.)
Therefore $\af_0$ descends to an automorphism $\af$ of
$C ([0, 1]) \star_{\mathrm{r}} C ([0, 1])
 \star_{\mathrm{r}} C ([0, 1])
 \star_{\mathrm{r}} {\mathbb{C}}^3$
such that $\alpha^3 = \Ad (u_0).$
\end{proof}

\begin{definition}\label{D:C0}
Let $A,$ $u_0 \in A,$ and $\af \in \Aut (A)$
be as in Definition~\ref{D:A0}.
Let $B$ and $v \in B$ be as in Definition~\ref{D:B0},
and let $\bt \in \Aut (B)$ be as in Lemma~\ref{lemma}.
We define $C = A \otimes B.$
(This tensor product is unambiguous since $B$ is nuclear.)
Choose and fix a unitary $w \in C^* (u_0 \otimes v) \subseteq C$
such that $w^3 = (u_0 \otimes v)^*$
(see Lemma~\ref{L:gm} below),
and define $\gm \in \Aut (C)$ by
$\gm = \Ad (w) \circ (\af \otimes \bt).$
We also write $\gm$ for the action of $\gruppo$ on~$C$
generated by $\Ad (w) \circ (\af \otimes \bt).$
(See Lemma~\ref{L:gm} below.)
\end{definition}

\begin{lemma}\label{L:gm}
Let the notation be as in Definition~\ref{D:C0}.
Then there is a unitary $w \in C^* (u_0 \otimes v) \subseteq C$
such that $w^3 = (u_0 \otimes v)^*,$
and the resulting $\gm \in \Aut (C)$ satisfies $\gm^3 = \id_{C}.$
\end{lemma}

\begin{proof}
Since
${\operatorname{sp}} (u_0), \, {\operatorname{sp}} (v)
 \subseteq \big\{ 1, \, \first, \, \third \big\},$
it immediately follows that
${\operatorname{sp}} (u_0 \otimes v)$ is finite.
Therefore there is a unitary $w \in C^* (u_0 \otimes v)$
such that $w^3 = (u_0 \otimes v)^*.$
Moreover, one checks that
$(\af \otimes \bt) (u_0 \otimes v) = u_0 \otimes v,$
whence $(\af \otimes \bt) (w) = w.$
It is now a routine calculation to check that
$\big( \Ad (w) \circ (\af \otimes \bt) \big)^3 = \id_{C}.$
\end{proof}

\begin{lemma}\label{L:Uniq}
The \ca\  $C$ of Definition~\ref{D:C0}
is exact and has a unique \tst, namely the tensor product \tst.
\end{lemma}

\begin{proof}
This is true for $A$ by Lemma~\ref{L:A0},
and for $B$ because $B$ is a UHF~algebra.
For exactness, apply Proposition~7.1(ii) of~\cite{Kr2},
and for existence of a unique \tst,
apply Corollary~6.13 of~\cite{CP}.
\end{proof}

\begin{definition}\label{D:D}
Define the \ca\  $D$ by $D = C \rtimes_{\gamma} \gruppo.$
\end{definition}

\section{The \ca\  $D = C \rtimes_{\gamma} \gruppo$
 is not anti-isomorphic to itself}\label{Sec:5}

\begin{definition}
Let $E$ be an arbitrary \ca.
The {\emph{opposite algebra}} $E^{\mathrm{op}}$ of~$E$
is by definition the \ca\  whose underlying Banach space
and adjoint operation
are the same as for~$E,$
while the product of $x$ by $y$ is equal to $y x$ instead of $x y.$
\end{definition}

Our main result is that the \ca\  %
$D = (A \otimes B) \rtimes_{\gamma} \gruppo,$
given in Definition~\ref{D:D},
is not isomorphic to its opposite algebra.
We reduce the statement about the nonisomorphism
of $D$ and $D^{\mathrm{op}}$
to a statement about the nonisomorphism
of the von Neumann algebra $M$ associated to $D$
(via the Gelfand-Naimark-Segal representation coming from
the obvious \tst\  of~$D$)
and $M^{\mathrm{op}}.$

To this end, we first examine the weak operator closures of the
images of $A$ and~$B$
in the Gelfand-Naimark-Segal representations
coming from their \tst s.
These algebras were defined so that the statements in
the following lemmas would hold.

\begin{lemma}\label{L:B0P}
Let $B$ be as in Definition~\ref{D:B0},
and let $\bt \in \Aut (B)$ be as in Lemma~\ref{lemma}.
Let $R_0$ be the weak operator closure of the image of~$B$
under the Gelfand-Naimark-Segal representation coming from
the unique \tst\  of~$B,$
and let ${\widetilde{\bt}}$ be the automorphism of~$R_0$
which extends~$\bt.$
Then $R_0$ is isomorphic to the hyperfinite factor $R$
of type~${\mathrm{II}}_1.$
The isomorphism $R_0 \cong R$ can be chosen so that
${\widetilde{\beta}}$ becomes the automorphism
of Proposition~1.6 of~\cite{Connes3},
for the case $p = 3$ and $\gm = \first.$
In particular, with $v$ as in Notation~\ref{N:B0},
we have ${\widetilde{\beta}}^3 = \Ad (v)$
and ${\widetilde{\beta}} (v) = e^{2 \pi i/3} v.$
\end{lemma}

\begin{proof}
This is obvious from the definitions
and from the proof of Proposition~1.6 of~\cite{Connes3}.
\end{proof}

\begin{lemma}\label{L:A0P}
Let $A$ be as in Definition~\ref{D:A0},
and let $\af \in \Aut (A)$ be as in Lemma~\ref{L:af}.
Let $N_0$ be the weak operator closure of the image of $A$
under the Gelfand-Naimark-Segal representation coming from
the unique \tst\  of~$A$ (Lemma~\ref{L:A0}),
and let ${\widetilde{\alpha}}$ be the automorphism of~$N_0$
which extends~$\af.$
Then $N_0$ is a factor of type~${\mathrm{II}}_1$ isomorphic to the free
product $N$ of three copies of $L^{\infty} ([0,1])$ (using
Lebesgue measure) and the group von Neumann algebra
${\mathcal{L}} (\gruppo) \cong \C^3$
(using the tracial state $\ph$ of Definition~\ref{D:A0}).
With
\[
\ld_k \colon L^{\infty} ([0,1])
 \to L^{\infty} ([0,1]) \star L^{\infty} ([0,1])
      \star L^{\infty} ([0,1]) \star {\mathcal{L}} (\gruppo),
\]
for $k = 1, 2, 3,$ being the inclusion of the $k$th free factor,
and with $u_0 \in \C^3 \cong {\mathcal{L}} (\gruppo)$ being as in
Definition~\ref{D:A0},
the isomorphism $N_0 \cong N$ can be chosen so that, on~$N,$
for all $g \in L^{\infty} ([0,1]),$ we have
\[
{\widetilde{\alpha}} (\lambda_{1} (g)) = \lambda_{2} (g),
\,\,\,\,\,\,
{\widetilde{\alpha}} (\lambda_{2} (g)) = \lambda_{3} (g),
\andeqn
{\widetilde{\alpha}} (\lambda_{3} (g)) = \Ad (u_0) (\lambda_{1} (g)),
\]
and ${\widetilde{\alpha}} (u_0) = \third u_0.$
Moreover, ${\widetilde{\alpha}}^3 = \Ad (u_0).$
\end{lemma}

\begin{proof}
This is obvious from the definitions.
\end{proof}

The nonisomorphism of the von Neumann algebra $M$
of the next proposition and its opposite
algebra was proved in~\cite{Viola}.
We recall it here for the convenience of the reader,
and refer to the original paper for the proof.

\begin{proposition}\label{P:V}
Let $N,$ $u_0,$ and ${\widetilde{\af}} \in \Aut (N)$
be as in Lemma~\ref{L:A0P},
and let $R,$ $v,$ and ${\widetilde{\beta}} \in \Aut (R)$
be as in Lemma~\ref{L:B0P}.
Let $w$ be as in Definition~\ref{D:C0},
regarded as an element of the von Neumann algebra
tensor product $N {\overline{\otimes}} R.$
Then
$\Ad (w)
 \circ \big( {\widetilde{\alpha}} \otimes {\widetilde{\beta}} \big)$
generates an action ${\widetilde{\gamma}}$ of $\gruppo$
on $N {\overline{\otimes}} R,$
and
$M = (N {\overline{\otimes}} R) \rtimes_{{\widetilde{\gamma}}} \gruppo$
is a factor of type~${\mathrm{II}}_1$
which is not isomorphic to its opposite von Neumann algebra.
\end{proposition}

\begin{proof}
That ${\widetilde{\gamma}}$ generates an action of $\gruppo$
follows as in the proof of Lemma~\ref{L:gm}.
(Or see~\cite{Viola}.)
That $M$ is a factor is in Section~4 of~\cite{Viola},
and $M \not\cong M^{\mathrm{op}}$ is Theorem~6.1 of~\cite{Viola}.
\end{proof}

\begin{proposition}\label{P:UniqTrD}
The \ca\  $D = C \rtimes_{\gamma} \gruppo$ of Definition~\ref{D:D}
is simple, separable, unital, exact, and has a unique tracial state.
\end{proposition}

\begin{proof}
It is obvious that $D$ is separable and unital.

The automorphism $\gm$ is outer,
because $\bt$ is outer (Lemma~\ref{lemma}).
Therefore $\gm^2 = \gm^{-1}$ is outer.
Simplicity of~$D$ now follows from Theorem~3.1 of~\cite{Ks1}.

Exactness follows from Proposition~7.1(v) of~\cite{Kr2}
and Lemma~\ref{L:Uniq}.

To prove that $D$ has a unique \tst,
apply Lemma~\ref{utrace},
using Lemmas \ref{L:Uniq}, \ref{L:B0P}, and~\ref{L:A0P},
as well as Proposition~\ref{P:V},
to verify its hypotheses.
\end{proof}

To reduce the nonisomorphism of the \ca\  $D$ and its opposite
to the nonisomorphism of the von Neumann algebra $M$ and its opposite,
we use the unique tracial state defined on $D$
and its associated Gelfand-Naimark-Segal representation.

\begin{theorem}\label{T:NotIsoOpp}
The \ca\  $D = C \rtimes_{\gamma} \gruppo$ of Definition~\ref{D:D}
is not isomorphic to its opposite algebra.
\end{theorem}

\begin{proof}
Assume that there exists an isomorphism
$\Phi \colon D \to D^{\mathrm{op}}.$
Let $\tau$ be the unique tracial state on $D$
(Proposition~\ref{P:UniqTrD}),
and denote by $\tau^{\mathrm{op}}$
the corresponding (unique) tracial state on $D^{\mathrm{op}}.$
Denote by $\pi$ and $\pi^{\mathrm{op}}$
the Gelfand-Naimark-Segal representations of $D$ and $D^{\mathrm{op}}$
associated to $\tau$ and $\tau^{\mathrm{op}}.$
By uniqueness of the tracial states,
we have $\tau  = \tau^{\mathrm{op}} \circ \Phi,$
so $\pi$ is unitarily equivalent to $\pi^{\mathrm{op}}\circ\Phi.$
It follows that as von Neumann algebras
$(\pi (D))^{\prime\prime}$
and $(\pi^{\mathrm{op}} (D^{\mathrm{op}}))^{\prime\prime}$
are isomorphic.
As in the proof of Theorem~3.2 of~\cite{Phil},
the algebra $(\pi (D))^{\prime\prime}$
is isomorphic to the factor $M$ of Proposition~\ref{P:V},
and
$(\pi^{\mathrm{op}} (D^{\mathrm{op}}))^{\prime\prime}
 \cong M^{\mathrm{op}}.$
So $M \cong M^{\mathrm{op}},$
contradicting Proposition~\ref{P:V}.
\end{proof}

\section{The Rokhlin property}\label{Sec:6}

We obtain further properties
of the \ca\  $D = C \rtimes_{\gamma} \gruppo$
by showing that $\gm$ has the Rokhlin property.

\begin{definition}\label{D:Rokhlin}
Let $E$ be a unital \ca\  and let $\sm \colon G \to \Aut (E)$
be an action of a finite group $G$ on~$E.$
We say that $\sigma$ has the {\emph{Rokhlin property}}
if for every finite set $F \subseteq E$ and every $\varepsilon > 0,$
there exist mutually orthogonal projections $e_g \in E$
for $g \in G$ such that:
\begin{enumerate}
\item\label{D:Rokhlin:1}
$\| \sigma_g (e_h) - e_{g h}\| < \varepsilon$ for all $g, h \in G.$
\item\label{D:Rokhlin:2}
$\|e_g a - a e_g\| < \varepsilon$ for all $g \in G$ and all $a \in F.$
\item\label{D:Rokhlin:3}
$\sum_{g \in G} e_g = 1.$
\end{enumerate}
\end{definition}

We first prove a Rokhlin-like property of order~$3$
for the automorphism $\bt$ of Lemma~\ref{lemma}.

\begin{lemma}\label{L:PartR}
Let $B$ be as in Definition~\ref{D:B0},
and let $\bt \in \Aut (B)$ be as in Lemma~\ref{lemma}.
Then for any $\varepsilon > 0$ and any finite set $F \subseteq B$
there exist projections $p_1,$ $p_2$ and $p_3$ such that:
\begin{enumerate}
\item\label{L:PartR:1}
$\| \beta (p_j) - p_{j + 1} \| < \varepsilon$ for $j = 1, 2,$
and
$\| \beta (p_3) - p_1\| < \varepsilon.$
\item\label{L:PartR:3}
$\|p_j a - a p_j\| < \varepsilon$
for all $a \in F$ and all $j \in \{1, 2, 3\}.$
\item\label{L:PartR:4}
$p_1 + p_2 + p_3 = 1.$
\end{enumerate}
\end{lemma}

\begin{proof}
Adopt the notation of Notation~\ref{N:B0} and Lemma~\ref{lemma}.
Since $B = \dirlim B_n,$
we need only consider finite subsets $F$ such that there is $n$
with $F \subseteq B_n.$
Set
\[
p_1 = \pi_{n + 1} (e_{1, 1}),
\,\,\,\,\,\,
p_2 = \pi_{n + 1} (e_{3, 3}),
\andeqn
p_3 = \pi_{n + 1} (e_{2, 2}).
\]
Obviously $p_1 + p_2 + p_3 = 1.$
Since $p_1,$ $p_2,$ and $p_3$ commute with every element in $F,$
condition~(\ref{L:PartR:3}) is verified.
To check condition~(\ref{L:PartR:1}),
we need to compute $\beta (p_j)$ for $j \in \{ 1, 2, 3 \}.$
Note that if $0 \leq k \leq n - 2$ or $k \geq n + 1,$
then $\lambda^k (u)$ commutes with $p_j$ for $j \in \{1, 2, 3\}.$
Set $v_0 = e^{2 \pi i/3} e_{1,1} + e^{4 \pi i/3} e_{2, 2} + e_{3, 3},$
so that $v = \pi_1 (v_0).$
We have
\begin{align*}
\Ad (\lambda^{n} (u)) (p_1)
&
 = \lambda^{n} \big( (e_{3, 1} \otimes v_0^* \otimes 1 \otimes \cdots)
 (e_{1, 1} \otimes 1 \otimes \cdots)
 (e_{3, 1} \otimes v_0^* \otimes 1 \otimes \cdots)^* \big)
= p_2, \\
\Ad (\lambda^{n} (u)) (p_2)
&
 = \lambda^{n} \big( (e_{2, 3} \otimes 1 \otimes 1 \otimes \cdots)
 (e_{3, 3} \otimes 1 \otimes \cdots)
 (e_{2, 3} \otimes 1 \otimes 1 \otimes \cdots)^* \big)
= p_3, \\
\Ad (\lambda^{n} (u)) (p_3)
&
 = \lambda^{n} \big( (e_{1, 2} \otimes 1 \otimes 1 \otimes \cdots)
 (e_{2, 2} \otimes 1 \otimes \cdots)
 (e_{1, 2} \otimes 1 \otimes 1 \otimes \cdots)^* \big)
= p_1.
\end{align*}
Also,
\[
\ld^{n - 1} (u) = \pi_n (e_{3, 1}) \pi_{n + 1} (v_0^*)
                     + \pi_n (e_{1, 2}) + \pi_n (e_{2, 3}),
\]
and $v_0^*$ commutes with $e_{1, 1},$ $e_{2, 2},$ and $e_{3, 3},$
so $\ld^{n - 1} (u)$ commutes with $p_1,$ $p_2,$ and~$p_3.$
Therefore
\begin{align*}
\beta (p_1) & = \Ad (\lambda^{n-1} (u)\lambda^n (u)) (p_1) = p_2, \\
\beta (p_2) & = \Ad (\lambda^{n-1} (u)\lambda^n (u)) (p_2) = p_3, \\
\beta (p_3) & = \Ad (\lambda^{n-1} (u)\lambda^n (u)) (p_3) = p_1.
\end{align*}
So $p_1,$ $p_2,$ and $p_3$ verify condition~(\ref{L:PartR:1}).
\end{proof}

\begin{proposition}\label{rokhlin}
The action $\gm \colon \gruppo \to \Aut (C)$ of Definition~\ref{D:C0}
has the Rokhlin property.
\end{proposition}

\begin{proof}
Recall that $\gm$ is generated by the automorphism
(also called $\gm$) $\Ad (w) \circ (\af \otimes \bt)$
of $C = A \otimes B.$

We first claim that for any finite subset $F \subseteq A \otimes B$
and any $\ep > 0,$
there are \pj s $e_1, e_2, e_3 \in A \otimes B$ such that:
\begin{enumerate}
\item\label{Pf:PartR:1}
$\| (\af \otimes \bt) (e_j) - e_{j + 1} \| < \varepsilon$
for $j = 1, 2,$
and
$\| (\af \otimes \bt) (e_3) - e_1 \| < \varepsilon.$
\item\label{Pf:PartR:2}
$\|e_j a - a e_j\| < \varepsilon$
for all $a \in F$ and all $j \in \{1, 2, 3\}.$
\item\label{Pf:PartR:3}
$e_1 + e_2 + e_3 = 1.$
\end{enumerate}
A standard approximation argument
shows that it suffices to choose a subset $G \subseteq A \otimes B$
which generates a dense *-subalgebra of~$A \otimes B,$
and consider only finite subsets $F \subseteq G.$
Thus, we may assume that there are finite subsets $S \subseteq A$
and $T \subseteq B$ such that
\[
F = \big\{ a \otimes b \colon {\mbox{$a \in S$ and $b \in T$}} \big\}.
\]
Set $M = \sup \{ \| a \| \colon a \in S \}.$
Apply Lemma~\ref{L:PartR} with $\ep / (M + 1)$ in place of~$\ep$
and with $T$ in place of~$F,$
obtaining \pj s $p_1, p_2, p_3 \in B$ as there.
Set $e_j = 1 \otimes p_j$ for $j \in \{ 1, 2, 3 \}.$
Conditions (\ref{Pf:PartR:1}) and~(\ref{Pf:PartR:3}) above
are immediate.
For~(\ref{Pf:PartR:2}),
let $a \in S$ and $b \in T,$ and let $j \in \{ 1, 2, 3 \}.$
Then
\[
\| (1 \otimes p_j) (a \otimes b) - (a \otimes b) (1\otimes p_j)\|
 = \| a \| \| p_j b - b p_j \|
 \leq M \cdot \ep / (M + 1)
 < \ep.
\]
This proves the claim.

To show that the action $\gamma = \Ad (w) \circ (\alpha \otimes \beta)$
has the Rokhlin property,
let $F \subseteq C$ be finite and let $\ep > 0.$
It suffices to find \pj s $e_1, e_2, e_3 \in C$ satisfying
Conditions (\ref{Pf:PartR:1})--(\ref{Pf:PartR:3}) in the claim above
with $\gm$ in place of $\af \otimes \bt.$
Choose $e_1, e_2, e_3 \in C$ as in the claim
with $F \cup \{ w \}$ in place of $F$
and $\tfrac{1}{2} \ep$ in place of $\ep.$
The analogs
of Conditions (\ref{Pf:PartR:2}) and~(\ref{Pf:PartR:3}) above
are immediate.
For~(\ref{Pf:PartR:1}),
estimate
\begin{align*}
\| \gamma (e_1) - e_2 \|
& = \|(\alpha \otimes \beta ) (e_1) - w^* e_2 w \|
 \leq \| (\alpha \otimes \beta) (e_1) - e_2 \|
         + \|e_2 - w^* e_2 w \|
 \\
& = \|(\alpha \otimes \beta) (e_1) - e_2 \| + \| w e_2 - e_2 w \|
 < \tfrac{1}{2} \ep + \tfrac{1}{2} \ep
 = \varepsilon.
\end{align*}
Similarly
$\| \gamma (e_2) - e_3 \| < \ep$ and $\| \gamma (e_3) - e_1 \| < \ep.$
\end{proof}


Recall the Jiang-Su algebra~$Z,$
from Theorem~2.9 of~\cite{JS}.
It is a simple separable unital nuclear \ca,
not of type~${\mathrm{I}},$
which has a unique tracial state
and which satisfies $K_0 (Z) \cong \Z$ and $K_1 (Z) = 0.$

\begin{proposition}\label{C:PropOfD}
The \ca\  $D = C \rtimes_{\gamma} \gruppo$
of Definition~\ref{D:D}
is approximately divisible, stably finite,
and has real rank zero and stable rank one.
The order on projections over $D$ is determined
by the unique tracial state on~$D,$
in the sense of Definition~\ref{D:OrdDetTr}.
Moreover, $D$ tensorially absorbs the $3^{\infty}$~UHF algebra~$B$
and the Jiang-Su algebra~$Z.$
\end{proposition}

\begin{proof}
The \ca\  $D$ is stably finite since it is simple and has
a (necessarily faithful) \tst.

The same applies to~$A.$
Since $B$ is a UHF~algebra,
Corollary~6.6 in~\cite{Rordam1}
implies that ${\operatorname{tsr}} (A \otimes B) = 1.$
Since $\gm$ has the Rokhlin property,
Proposition~4.1(1) in~\cite{OsaPhil} now implies that
${\operatorname{tsr}} (D) = 1.$

Since any quasitrace
on a unital exact \ca\  is a tracial state~(\cite{Haa}),
Lemma~\ref{L:A0} implies that the \pj s in $A$ distinguish
the quasitraces.
Theorem~7.2 in~\cite{Rordam2}
now implies that $A \otimes B$ has real rank zero,
and Proposition~4.1(2) in~\cite{OsaPhil}
implies that $D$ has real rank zero.

Lemma~\ref{L:AppDivAlgs}(\ref{L:AppDivAlgs:1})
implies that $B$ is approximately divisible,
so Lemma~\ref{L:AppDivAlgs}(\ref{L:AppDivAlgs:2})
implies that $A \otimes B$ is approximately divisible,
and Proposition~4.5 in~\cite{OsaPhil}
(or Corollary~3.4(2) in~\cite{HW})
implies that $D$ is approximately divisible.

Proposition~\ref{order} now implies that
the order on projections over $D$ is determined by traces.

For the absorption properties,
first note that
$C \otimes B = A \otimes B \otimes B \cong A \otimes B.$
Thus $C$ absorbs~$B.$
It now follows from Corollary 3.4(1) of~\cite{HW}
that $D$ absorbs~$B.$
Now $D$ absorbs $Z$ because,
by Corollary~6.3 of~\cite{JS}, the algebra $B$ absorbs~$Z.$
\end{proof}

\section{The K-theory of the crossed
 product $D = C \rtimes_{\gamma} \gruppo$}\label{Sec:7}

In this section, we compute the K-theory of the algebra
\[
D = C \rtimes_{\gamma} \gruppo
 = (A \otimes B) \rtimes_{\Ad (w) \circ (\af \otimes \bt)} \gruppo.
\]
We start with the K-theory of $A.$

\begin{lemma}\label{KA}
Let $A$ be as in Definition~\ref{D:A0}.
Then the inclusion $\et \colon \C^3 \to A$ of the last free factor
is a KK-equivalence.
In particular,
$K_1 (A) = 0$ and $K_0 (A) \cong {\mathbb{Z}}^3,$
generated by the classes of the images of the \pj s
\[
r_1 = (1, 0, 0), \,\,\,\,\,\,
r_2 = (0, 1, 0),
\andeqn
r_3 = (0, 0, 1)
\]
in~$\C^3.$
Moreover, $K_0 (\af)$ is given by permuting the coordinates:
$(\et_1, \et_2, \et_3) \mapsto (\et_3, \et_1, \et_2).$
\end{lemma}

\begin{proof}
First, the inclusion $\C \to C ([0, 1])$ is clearly a homotopy
equivalence.
By taking the full free product of homotopies,
we find that
\[
\C = \C \star \C \star \C
 \to C ([0, 1]) \star C ([0, 1]) \star C ([0, 1])
\]
(full free products)
is a homotopy equivalence.
Therefore
\[
\C^3 = \C \star \C^3
    \to C ([0, 1]) \star C ([0, 1]) \star C ([0, 1]) \star \C^3
\]
is a homotopy equivalence.
Theorem~4.1 of of~\cite{Germain}
(see the introduction to that paper for the notation $A$ and $A_r$
used there)
implies that
\[
C ([0, 1]) \star C ([0, 1]) \star C ([0, 1]) \star \C^3
 \to C ([0, 1]) \star_{\mathrm{r}} C ([0, 1])
    \star_{\mathrm{r}} C ([0, 1]) \star_{\mathrm{r}} \C^3
 = A
\]
is a KK-equivalence.
(Note that ``K-equivalence'' in~\cite{Germain} is what is usually
called KK-equivalence; see Section~6 of~\cite{Sk}.)

To compute $K_0 (\af),$
observe that the unitary $u_0$ of Definition~\ref{D:A0} is given by
$u_0 = \first r_1 + r_2 + \third r_3.$
Since it generates $\C^3$ and
\[
\af (u_0) = \third u_0 = r_1 + \third r_2 + \first r_3,
\]
we must have
$\af (r_1) = r_3,$ $\af (r_2) = r_1,$ and $\af (r_3) = r_2.$
The desired formula is now immediate.
\end{proof}

\begin{proposition}\label{P:7.4}
Let $C = A \otimes B$ be as in Definition~\ref{D:C0},
and let $D = C \rtimes_{\gamma} \gruppo$
be as in Definition~\ref{D:D}.
Then
\[
K_0 (D) \cong {\mathbb{Z}} \big[ \tfrac{1}{3} \big]
\andeqn
K_1 (D) = 0.
\]
The first isomorphism sends $[1]$ to~$1,$
and is an isomorphism of ordered groups.
Letting $\ta$ be the \tst\  on $D$ as in Proposition~\ref{P:UniqTrD},
the map $\ta_* \colon K_0 (D) \to \R$ corresponds to the
inclusion of ${\mathbb{Z}} \big[ \tfrac{1}{3} \big]$ in~$\R.$
\end{proposition}

\begin{proof}
We use the notation of Definition~\ref{D:B0}, Lemma~\ref{lemma},
Definition~\ref{D:A0}, Lemma~\ref{L:af},
and Definition~\ref{D:C0}.

We begin by computing $K_* (C)$ and $K_* (\gm).$
We apply the K\"{u}nneth formula~\cite{Sc2};
this is valid because $B$ is in the bootstrap category of~\cite{Sc2}.
Since $K_0 (B) \cong {\mathbb{Z}} \big[ \tfrac{1}{3} \big],$
with $[1]$ being sent to~$1,$
and
$K_1 (B) = 0,$
Lemma~\ref{KA} gives
\[
K_1 (C) = 0
\andeqn
K_0 (C) \cong {\mathbb{Z}} \big[ \tfrac{1}{3} \big]^3.
\]
The inclusion $\C^3 \to A$ of the last free factor in~$A$
gives an inclusion $\C^3 \otimes B \to C,$
and this map is an isomorphism on K-theory.
Since $K_* (\bt)$ is the identity,
it moreover follows that
$K_0 (\af \otimes \bt) \colon {\mathbb{Z}} \big[ \tfrac{1}{3} \big]^3
   \to {\mathbb{Z}} \big[ \tfrac{1}{3} \big]^3$
is given by permuting the coordinates as in Lemma~\ref{KA}:
$(\et_1, \et_2, \et_3) \mapsto (\et_3, \et_1, \et_2).$
Since $\Ad (w)$ is trivial on K-theory,
$K_0 (\gm)$ is given by the same formula.

We now compute the K-theory of the fixed point
algebra $C^{\gm}$ of $C$ under the action~$\gm.$
Since $C$ is simple and unital,
and $\gm$ has the Rokhlin property (Proposition~\ref{rokhlin}),
we can apply Theorem~3.13 in~\cite{Izumi} to conclude
that the inclusion of $C^{\gm}$ in $C$ is injective on K-theory,
and that its range is
\[
\bigcap_{m = 0}^2 \ker (\id - K_* (\gm^m)).
\]
We can ignore the term for $m = 0.$
Since
\[
\id - K_* (\gm^2) = K_* (\gm^2) \circ (- (\id - K_* (\gm)))
\]
and $K_* (\gm^2)$ is injective,
we can also ignore the term for $m = 2.$
Thus, $K_* (C^{\gm}) \cong \ker (\id - K_* (\gm)).$

It immediately follows that $K_1 (C^{\gm}) = 0.$
Moreover, $\id - K_0 (\gm)$ is given by the matrix
\[
\left( \begin{array}{rrr}
    1 & -1 & 0 \\
    0 &  1 & -1 \\
   -1 &  0 & 1
\end{array} \right).
\]
The map $\et \mapsto (\et, \et, \et)$ is an isomorphism from
${\mathbb{Z}} \big[ \tfrac{1}{3} \big]$ to $\ker ( \id - K_0 (\gm) )$
which sends $1$ to $(1, 1, 1) = [1].$
This completes the computation of $K_* (C^{\gm}).$

Let $z \in C \rtimes_{\gamma} \gruppo$ be the standard
unitary corresponding to the usual generator of $\gruppo,$
and let $p \in C \rtimes_{\gamma} \gruppo$ be the \pj\  %
$p = \tfrac{1}{3} (1 + z + z^2).$
The Proposition, Corollary, and proof of the Corollary in~\cite{Rs}
imply that $C^{\gm}$
is isomorphic to the corner $p (C \rtimes_{\gamma} \gruppo) p.$
Since $C \rtimes_{\gamma} \gruppo$ is simple,
this corner is full,
and its inclusion is an isomorphism on K-theory.
Thus,
\[
K_1 (C \rtimes_{\gamma} \gruppo) = 0
\andeqn
K_0 (C \rtimes_{\gamma} \gruppo)
   \cong {\mathbb{Z}} \big[ \tfrac{1}{3} \big],
\]
with $[p]$ corresponding
to~$1 \in {\mathbb{Z}} \big[ \tfrac{1}{3} \big].$
Let
${\widehat{\gm}} \in \Aut (C \rtimes_{\gamma} \gruppo)$
denote the generator of the dual action.
Then $p + {\widehat{\gm}} (p) + {\widehat{\gm}}^2 (p) = 1.$
Moreover, $K_0 ( {\widehat{\gm}} )$
is an automorphism of ${\mathbb{Z}} \big[ \tfrac{1}{3} \big]$
which fixes the nonzero element~$[1],$
so we must have $K_0 ( {\widehat{\gm}} ) = \id.$
Therefore $[1] = 3 [p].$
Since multiplication by~$3$
is an automorphism of ${\mathbb{Z}} \big[ \tfrac{1}{3} \big],$
it follows that there is an isomorphism
$K_0 (C \rtimes_{\gamma} \gruppo)
   \cong {\mathbb{Z}} \big[ \tfrac{1}{3} \big]$
which sends $[1]$ to~$1 \in {\mathbb{Z}} \big[ \tfrac{1}{3} \big].$

The map $\ta_* \colon K_0 (D) \to \R$ is a \hm\  %
which sends $[1]$ to~$1 \in \R.$
There is a unique \hm\  ${\mathbb{Z}} \big[ \tfrac{1}{3} \big] \to \R$
which sends $1$ to~$1,$
namely the inclusion.
Therefore the isomorphism
$K_0 (C \rtimes_{\gamma} \gruppo)
   \cong {\mathbb{Z}} \big[ \tfrac{1}{3} \big]$
must send $\ta_*$ to the inclusion.
Since the order on \pj s over~$D$ is determined by
the \tst s (Proposition~\ref{C:PropOfD}),
it follows that this isomorphism is an order isomorphism
for the usual order on ${\mathbb{Z}} \big[ \tfrac{1}{3} \big].$
\end{proof}

Since $D$ is simple and tensorially absorbs the Jiang-Su algebra,
we can also determine the Cuntz semigroup of the \ca~$D.$
(See Definitions \ref{D:CuEq} and~\ref{D:CuSg}.)

\begin{corollary}\label{Cuntz}
Let $D = C \rtimes_{\gamma} \gruppo$
be as in Definition~\ref{D:D}.
The Cuntz semigroup of~$D$ is given by
$W (D)
 \cong {\mathbb{Z}} \big[ \tfrac{1}{3} \big]_{+} \sqcup (0, \infty).$
Here, ${\mathbb{Z}} \big[ \tfrac{1}{3} \big]_{+}$ is the set of
nonnegative elements in ${\mathbb{Z}} \big[ \tfrac{1}{3} \big].$
Addition and the order on each part of the disjoint union
are the usual ones.
If $\et \in {\mathbb{Z}} \big[ \tfrac{1}{3} \big]_{+}$
and $\et_0 \in (0, \infty)$ is the corresponding
element of $(0, \infty),$
and if $\mu \in (0, \infty),$
then $\et \leq \mu$
if and only if $\et_0 < \mu$ in $(0, \infty),$
while $\et \geq \mu$
if and only if $\et_0 \geq \mu$
in $(0, \infty).$
Moreover, $\et + \mu = \et_0 + \mu,$
the right hand side being computed in $(0, \infty).$
\end{corollary}

\begin{proof}
Since $D$ is simple, unital, exact, stably finite,
and $Z$-stable (Proposition~\ref{P:UniqTrD} and Proposition~\ref{C:PropOfD}),
we can apply Corollary~5.7 of~\cite{BPT}.
Also, the order on \pj s over~$D$ is determined by
the \tst s (Proposition~\ref{C:PropOfD}), so the K-theory computation (Proposition~\ref{P:7.4}) implies that $V (D) \cong {\mathbb{Z}} \big[ \tfrac{1}{3} \big]_{+}$.
Lastly, $\operatorname{LAff}_b (T(D))^{++}$
(as defined in Section~2 of~\cite{BPT})
is isomorphic to $(0, \infty),$
since $T (D)$ consists of just one point.
The proof is completed by observing that the order
and semigroup operation in the
statement match those defined on ${\widetilde{W}} (D)$ in
Section~2 of~\cite{BPT}.
\end{proof}

\section{Open problems}\label{Sec:8}

We discuss here some open questions on
simple separable \ca s not isomorphic to their opposites.
Similar techniques to those of this paper give examples with
some other choices of K-theory.
(Details will appear elsewhere.)
However, it seems that new methods are required to solve
the general case of the following problem.

\begin{question}\label{P:ArbK}
Let $B$ be any UHF~algebra.
Is there a simple separable exact \ca,
not isomorphic to its opposite algebra,
whose K-theory is the same as that of~$B,$
but which otherwise has all the properties of the algebra~$D$
constructed in this paper?
\end{question}

Of course, one can generalize this question,
letting $B$ be a simple unital AF~algebra,
or a simple unital AH~algebra with no dimension growth
and real rank zero.
If one drops the requirement that $D$ have real rank zero,
there are even more choices for~$B.$

New methods are also needed to address the following two questions.

\begin{question}\label{P:PI}
Is there a simple separable purely infinite \ca\  %
which is not isomorphic to its opposite algebra?
\end{question}

\begin{question}\label{P:Nuc}
Is there a simple separable nuclear \ca\  %
which is not isomorphic to its opposite algebra?
\end{question}

We see no obvious obstruction to a positive answer to
Question~\ref{P:PI},
especially since there are type~${\mathrm{III}}$ factors not isomorphic
to their opposite algebras~\cite{Connes2}.
A positive answer to Question~\ref{P:Nuc}
would be much more surprising,
in view of the Elliott program and the fact that all
known invariants of simple nuclear \ca s,
even the Cuntz semigroup,
are unable to distinguish a \ca\  from its opposite.

\section*{Acknowledgments}

The second author would like to thank Prof.\   E.~Kirchberg
for useful discussions,
and for suggestions helpful in proving Lemma~\ref{utrace}.

Some of this work was carried out during the
Fields Institute program on operator algebras of Fall 2007,
and during a visit by the first author to the Fields Institute
for the Spring 2008 Canadian Operator Symposium.
Both authors are grateful to the Fields Institute for
its hospitality.

\end{document}